\documentclass[11pt,reqno]{amsart}
  \usepackage{amsmath, amsfonts, amssymb,amsthm}

\newtheorem{theorem}{Theorem}
\newtheorem{lemma}{Lemma}

\theoremstyle{definition}

\newtheorem{defi}[theorem]{Definition}
\newtheorem{remark}[theorem]{Remark}

\def\min{\mathop{\mathrm{min}}}

\newcommand{\N}{\mathbb{N}}

\begin{document}
\title{{A Schmidt's subspace theorem  for moving hyeprplane targets over function fields}}
\author{Le Giang}
\address{Department of Mathematics, Hanoi National University of Education, 136 Xuan Thuy Street, Cau Giay, Hanoi, Viet Nam.}
\email{legiang01@yahoo.com}

\author{Tran Van Tan$^*$}
\address{Department of Mathematics, Hanoi National University of Education, 136 Xuan Thuy Street, Cau Giay, Hanoi, Viet Nam.}
\email{tranvantanhn@yahoo.com}

\author{Nguyen Van Thin}
\address{Department of Mathematics, Thai Nguyen University of Education, Luong Ngoc Quyen Street, Thai Nguyen city, Viet Nam.}
\email{thinmath@gmail.com}

\thanks{$^*$ Corresponding author.}

\begin{abstract}
In this paper, we establish a Schmidt's subspace theorem for moving hyeprplane targets   in projective spaces over function fields.

\end{abstract}
\baselineskip=16truept 
\maketitle 
\pagestyle{myheadings}
\markboth{}{}
\section{ Introduction}
\def\theequation{1.\arabic{equation}}
\setcounter{equation}{0}

First of all, we remind some notation and fundamental results of  Diophantine Approximation  for moving hyperplanes (see \cite{RV, ZMQ}).
\begin{defi}\label{defi1.1}
Let $\Lambda$ be an infinite index set.

$(i)$ A {\it moving hyperplane} indexed by $\Lambda$ is a map $H: \Lambda \to \mathbb{P}^{M}(k)^{*}$, defined by $\alpha \mapsto H(\alpha).$

$(ii)$ Let $H_1, \dots, H_q$ be moving hyperplane indexed by $\Lambda$. For each $j=1, \dots, q$ and $\alpha \in \Lambda$ choose $a_{j,0}(\alpha), \dots, a_{j,M}(\alpha) \in k$ such that $H_j(\alpha)$ is cut out by the linear form $L_j(\alpha)=a_{j,0}(\alpha)x_0+\dots+a_{j,M}(\alpha)x_M.$ A subset $A \subset \Lambda$ is said to be {\it coherent} with respect to $\{H_1, \dots, H_q\}$ if for every polynomial
$$ P\in k[x_{1,0}, \dots, x_{1,M}, \dots, x_{q,0}, \dots, x_{q,M}] $$
that is homogeneous in $x_{j,0}, \dots, x_{j,M}$ for each $j=1, \dots, q$, either
$$ P(a_{1,0}(\alpha), \dots, a_{1,M}(\alpha), \dots, a_{q,0}(\alpha), \dots, a_{q,M}(\alpha)) $$
vanishes for all $\alpha \in A$ or it vanishes for only finitely many $\alpha \in A.$

$(iii)$ For a subset $A\subset \Lambda$, we define $\mathcal R_{A}^{0}$ to be the set of equivalence classes of pairs $(C, a)$, where $C\subset A$ is a subset with finite complement; $a: C \to k$ is a map; and the equivalence relation is defined by $(C, a) \sim (C', a')$
if there exists $C'' \subset C\cap C'$ such that $C''$ has finite complement in $A$ and $a|_{C''}=a'|_{C''}.$ This is a ring containing $k$ as a subring.

$(iv)$  Let $\mathcal H=\{H_1, \dots, H_q\}$ be a set of moving hyperplanes. If $A$ is coherent with respect to $\mathcal H$, and if
 $a_{j,t}(\alpha)\ne 0$ for all, but finitely many $\alpha \in A$, then $\dfrac{a_{j,l}}{a_{j,t}}$ defines an element of 
$\mathcal R_{A}^{0}.$ Moreover, by coherence, the subring of $\mathcal R_{A}^{0}$ generated by all such elements is an 
integral domain. We define $\mathcal R_{A, \mathcal H}$ to be the field of fractions of that integral domain.
\end{defi}

Let $k$ be an algebraic number field of degree $d$. Denote by $M(k)$ for the set of places (equivalence classes of absolute values) of $k$ and write
$M_{\infty}(k)$ for the set of archimedean places of  $k$. For each $v \in M(k)$, we choose a normalized absolute value $| . |_{v}$ such that
$| p |_{v}=| . |$ on $\mathbb{Q}$ (the standard absolute value) if $v$ is archimedean, whereas for $v$ non-archimedean, $|p|_{v}=p^{-1}$
if $v$ lies above the rational prime $p$. Denote by $k_v$ for the completion of $k$ with respect to $v$ and by $d_v=[k_v: Q_v]$ for the local 
degree. We put $\Vert.\Vert_{v}=| . |_{v}^{d_v/d}$. Then norm $||.||_v$ satisfies the following properties:

$(i)$ $||x||_v\ge 0$, with equality if and only if $x=0;$

$(ii)$ $||xy||_v=||x||_v\cdot ||y||_v$ for all $x, y \in k;$

$(iii)$ $||x_1+\dots+x_n||_v\le B_v^{n_v}\cdot \max \{||x_1||_v, \dots, ||x_n||_v\}$ for all $x_1, \dots, x_n \in k$, $n\in \N$, where $n_v=d_v/d$, $B_v=1$ if $v$ is non-archimedean and $B_v=n$ if $v$ is archimedean. 

 Moreover, for each $x\in k\setminus \{0\}$, we have the {\it product formula}:
$$ \prod_{v\in M(k)}\Vert x\Vert_{v}=1. $$
For $x\in k$, we
 define the {\it logarithmic height} of $x$ by 
 $$h(x):=\sum_{v\in M(k)} \log^{+}\Vert x\Vert_{v},$$ where 
$\log^{+}\Vert x\Vert_{v}= \log \max \{\Vert x\Vert_{v}, 1\}.$

 For $x=[x_0: \dots : x_M] \in \mathbb P^M(k)$, let
 $\Vert x\Vert_{v}=\max_{0 \le i \le M}\Vert x_i\Vert_{v}.$ We define the {\it logarithmic height} of $x$ by
\begin{align}\label{ct11}
h(x):=\sum_{v\in M(k)} \log \Vert x\Vert_{v}.
\end{align}
By the product formula, this definition is independent of the choice of the representations.

 Let $H\subset P^M(k)$ be a (fixed) hyperplane, and let $L=a_0x_0+\dots+a_Mx_M$ be linear form defining $H$ with $a_0, \dots, a_M \in k.$ Take $\Vert L\Vert_{v}=\max_{0 \le i \le M}\Vert a_i\Vert_{v}$. We also define the {\it height} of $H$ as following
$$ h(H) :=h([a_0: \dots : a_M])=\sum_{v\in M(k)}\log \Vert L\Vert_{v}.$$
The {\it Weil function of a point $x$ ($x\not\in H$) with respect to the hyperplane H}  at a place $v\in M(k)$ is defined by
$$ \lambda_{H, v}(x):=\log \dfrac{\Vert x\Vert_{v}\cdot \Vert L\Vert_{v}}{\Vert a_0x_0+\dots a_Mx_M\Vert_{v}}.$$
Let $ S \subset M(k)$ be a finite set containing all archimedean places. The proximitiy function $m_{S}(H, x)$ is defined by
$$ m_{S}(H, x):=\sum_{v\in S} \lambda_{H,v}(x).$$
The following result is due to Ru and Vojta \cite{RV}: 

\noindent{\bf  Theorem A (Schmidt's subspace theorem with moving hyperplane).} {\it Let $k$ be a number field and let $S \subset M(k)$ be a
 finite set containing all archimedian places. Let $\Lambda$ be an infinite index set and let $\mathcal H:=\{H_1, \dots, H_q\}$  be a set of moving 
hyperplane indexed by $\Lambda$ in $\mathbb P^{M}(k).$ Let $x=[x_0: \dots: x_M]: \Lambda \to \mathbb P^{M}(k)$ be a sequence 
of points. Assume that

$(i)$ for all $\alpha \in \Lambda$, $H_1(\alpha), \dots, H_q(\alpha)$ are in general position;

$(ii)$ $x$ is linearly non-degenerate with respect to $\mathcal H$, which means that for each infinite coherent subset $A \subset \Lambda$ with respect to $\mathcal H$, $x_0|_{A}, \dots, x_M|_{A}$ are linearly independent over $\mathcal R_{A,\mathcal H};$

$(iii)$ $h(H_j(\alpha))=o(h(x(\alpha)))$ for all $j=1, \dots , q$ (i.e for any $\delta >0$, $h(H_j(\alpha))\le \delta h(x(\alpha))$ for all but finitely many $\alpha \in A$).

Then, for any $\varepsilon >0,$ there exists an infinite index subset $A\subset \Lambda$ such that
$$ \sum_{j=1}^{q}m_{S}(H_j(\alpha), x(\alpha)) \le (M+1+\varepsilon)h(x(\alpha))$$
holds for all $\alpha \in A.$
}

The purpose of this paper is to extend the above result of Ru-Vojta to the case of  function fields. In order to
 establish the main theorem, we remind some notation and fundamental results of Diophantine Approximation  over function fields in \cite{RW}.

Let $k$ be  an algebraically closed field of characteristic $0$ and $V$ be a non-singular projective variety. We shall fix a projective embedding of $V$ in $\mathbb P^L.$ Denote by $K=k(V )$ the function field of $V$. Let $M_K$ denote the set of prime divisors of $V$(irreducible subvarieties of codimension one). Let $\bf \mathfrak p$ $ \in M_K$ be a prime divisor. As V is nonsingular, the local ring $\mathcal O_p$ at $\bf \mathfrak p$ is a discrete valuation ring. For each $x\in K^{*}$, its order $\textup{ord}_{\bf \mathfrak p}x$ at $\bf \mathfrak p$ is well defined. We can associate to $x$ its divisor
$$ (x) =\sum_{{\bf \mathfrak p} \in M_K}\textup{ord}_{\bf \mathfrak p}(x) {\bf \mathfrak p}=(x)_0-(x)_{\infty},$$
where $(x)_0$ is the zero divisor of $(x)$ and $(x)_{\infty}$ is the polar divisor of $x$, respectively. Let $\deg \bf \mathfrak p$ denote the projective degree of $\bf \mathfrak p$ in $\mathbb P^{L}.$ Then we have the sum formula
$$ \deg (x) =\sum_{{\bf \mathfrak p}\in M_K} {\textup{ord}}_{\bf \mathfrak p}(x)\deg {\bf\mathfrak p}=0. $$
Let ${\bf x}=[x_0, \dots, x_M]\in \mathbb P^M(K)$, denote by
$$ e_{\bf \mathfrak p}({\bf x}):=\min _{0\le i\le M}\{\textup{ord}_{\bf \mathfrak p}(x_i)\}.$$
The (logarithmic) height of ${\bf x}$ is defined by
$$ h({\bf x}):=-\sum_{{\bf \mathfrak p}\in M_K}e_{\bf \mathfrak p}(x)\deg {\bf\mathfrak p}. $$
A (fixed) hyperplane $H$ in $\mathbb P^M(K)$ is the zero set of a linear form  $L(X_0, \dots, X_M)=\sum_{j=0}^{M}a_jX_j,$ 
where $a_j\in K$, $j=0, \dots, M$.
Denote by
$$ e_{\bf \mathfrak p}(L):=\min_{0\le j\le M} \{{\textup{ord}}_{\bf \mathfrak p}(a_{j})\}.$$
The (logarithmic) height of $H$ is defined by
$$ h(H):=-\sum_{{\bf \mathfrak p}\in M_K}e_{\bf \mathfrak p}(L)\deg {\bf\mathfrak p}. $$
The Weil function of $\lambda_{{\bf \mathfrak p}, H}$ at ${\bf \mathfrak p}$ is defined by the following formula:
\begin{align}\label{ct11a}
\lambda_{{\bf \mathfrak p}, H}({\bf x}):=(\textup{ord}_{\bf \mathfrak p}L({\bf x})-e_{\bf \mathfrak p}({\bf x})-
e_{\bf \mathfrak p}(L)) \deg {\bf \mathfrak p}\ge 0,
\end{align}
where ${\bf x}=[x_0, \dots, x_M] \in \mathbb P^{M}(K) \setminus H.$

\begin{defi}\label{defi1.1a}
Let $\Lambda$ be an infinite index set.

$(i)$ A {\it moving hyperplane} over function field $K$ indexed by $\Lambda$ is a map $H: \Lambda \to \mathbb{P}^M(K)^{*}$, defined by $\alpha \mapsto H(\alpha).$

$(ii)$ Let $H_1, \dots, H_q$ be moving hyperplane indexed by $\Lambda$. For each $j=1, \dots, q$ and $\alpha \in \Lambda$
 choose $a_{j,0}(\alpha), \dots, a_{j,M}(\alpha) \in K$ such that $H_j(\alpha)$ is cut out by the linear form
 $L_j(\alpha)=a_{j,0}(\alpha)x_0+\dots+a_{j,M}(\alpha)x_M.$ A subset $A \subset \Lambda$ is said to be {\it coherent} with respect to $\{H_1, \dots, H_q\}$ if  for every polynomial
$$ P\in K[X_{1,0}, \dots, X_{1,M}, \dots, X_{q,0}, \dots, X_{q,M}] $$
that is homogeneous in $X_{j,0}, \dots, X_{j,M}$ for each $j=1, \dots, q$, either
$$ P(a_{1,0}(\alpha), \dots, a_{1,M}(\alpha), \dots, a_{q,0}(\alpha), \dots, a_{q,M}(\alpha)) $$
vanishes for all $\alpha \in A$ or it vanishes for only finitely many $\alpha \in A.$
\end{defi}
\begin{remark} The above definition is independent of the choice of coefficients $a_{j,0}(\alpha), a_{j,M}(\alpha)$.
\end{remark}
\begin{lemma} There exists an infinite subset $A\in\Lambda$ which is coherent with respect to $(H_1,\ldots, H_q)$.
\end{lemma}
The proof of this lemma is similar to the proof of Lemma 1.1 in \cite{RV} without any modifications.
\begin{defi} Let $\lambda$ be an infinite index set.

$(i)$ For a subset $A\subset \Lambda$, we define $\mathcal R_{A}^{0}$ to be the set of equivalence classes of pairs $(C, a)$, where $C\subset A$ is a subset with finite complement and $a: C \to k$ is a map; and the equivalence relation is defined by: $(C, a) \sim (C', a')$
if there exists $C'' \subset C\cap C'$ such that $C''$ has finite complement in $A$ and $a|_{C''}=a'|_{C''}.$ This is a ring containing 
$K$ as a subring.

$(ii)$  Let $\mathcal H=\{H_1, \dots, H_q\}$ be a set of moving hyperplanes. If $A$ is coherent with respect to $\mathcal H$, and if
 $a_{j,t}(\alpha)\ne 0$ for all but finitely many $\alpha \in A$, then $\dfrac{a_{j,l}}{a_{j,t}}$ defines an element of 
$\mathcal R_{A}^{0}.$ Moreover, by coherence, the subring of $\mathcal R_{A}^{0}$ generated by all such elements in an integral
 domain. We define $\mathcal R_{A, \mathcal H}$ to be the field of fractions of that integral domain.
\end{defi}
\begin{remark}{\label{re:1.3}} Let $B\subset A\subset\Lambda$ be two infinite index subsets. Then it's clear that if $A$ is coherent then so is $B$, and $\mathcal{R}_{A,\mathcal{H}}=\mathcal{R}_{B,\mathcal{H}}.$ 
\end{remark}
\begin{defi}\label{defi1.1b} Let $\mathcal H=\{H_1, \dots, H_q\}$ be a set of moving hyperplanes indexed by $\Lambda.$ 
  Let ${\bf x}=[x_0: \dots; x_M]: \Lambda \to \mathbb P^{M}(K)$ be a sequences of points. We say that $x$ is {\it nondegenerate 
respect to } $\mathcal H$ if for each infinite coherent subset $A\subset \Lambda$, the $x_0|_{A}, \dots, x_n|_{A}$ are linearly independent
over $\mathcal R_{A}.$
\end{defi}
Let $x$ be a map $x: \Lambda\longrightarrow\mathbb{P}^M(K)$. A map $(C,a)\in\mathcal{R}^0_{\Lambda}$ is called ``small" with respect to $x$ iff 
$$h(a(\alpha))=o(h(x(\alpha))),$$that means, for every $\epsilon>0$, there exists a subset $C_{\epsilon}\subset C$ with finite complement such that $h(a(\alpha))\leq\epsilon h(x(\alpha))$ for all $\,\alpha\in C_{\epsilon}$.  It is obvious that
$$|ord_\mathfrak{p}(a(\alpha))|=o(h(x(\alpha))).$$
Denote by $\mathcal{K}_x$ the set of all ``small" maps. Then, $\mathcal{K}_x$ is a subring of $\mathcal{R}^0_{\Lambda}.$
 But it's not an entire ring. However, if $(C,a)\in\mathcal{K}_x$ and $a(\alpha)\not=0$ for all but finitely $\alpha\in C$ then we have $\left(C\backslash\{a(\alpha)=0\}, \dfrac{1}{a}\right)\in\mathcal{K}_x$.

 
 
 %

Now, our Schmidt's subspace theorem with moving hyperplanes over function fields is stated as follows:

\begin{theorem}\label{th1}  Let $K$ be the function field of a nonsingular variety $V$ defined over an algebraically closed 
field $k$ of characteristic zero. Let $S$ be a finite set of prime divisors of $V$.  Let $\Lambda$ be an infinite index set and let $H_j$ be
moving  hyperplanes in $\mathbb P^{M}(K)$, $j=1, \dots, q,$ indexed by $\Lambda.$  Let ${\bf x}=[x_0: \dots: x_M]: \Lambda \to 
\mathbb P^{M}(K)$ be a sequence of points. Assume that

$(i)$ $H_1(\alpha), \dots, H_q(\alpha)$ are in general position  for each $\alpha \in \Lambda$; 

$(ii)$ ${\bf x}$ is linearly nondegenerate with respect to $\mathcal H=\{H_1, \dots, H_q\}$;
 
$(iii)$ $h(H_j(\alpha))=o(h(x(\alpha)))$ for all $j=1, \dots , q$.

Then, for any $\varepsilon >0,$ there exists an infinite index subset $A\subset \Lambda$ such that
$$ \sum_{j=1}^{q} \sum_{{\bf \mathfrak p}\in S}\lambda_{{\bf \mathfrak p}, H_j(\alpha)}({\bf x}(\alpha))\le (M+1+\varepsilon)
h({\bf x}(\alpha))+O(1).$$
holds for all $\alpha \in A.$
\end{theorem}

The condition (iii) means that for any given $\delta>0,$ we have $h(H_j(\alpha))\le \delta h(x(\alpha))$ for all but finitely many $\alpha\in A.$ 

\noindent {\bf Acknowledgements:} The second  named author was supported by a grant of Vietnam National Foundation for Science and Technology Development (NAFOSTED)  and the Vietnam Institute for Advanced Studies in Mathematics. He  is currently Regular Associate Member of ICTP, Trieste, Italy.
\section{Proof of Theorem \ref{th1}  }
\def\theequation{3.\arabic{equation}}
In order to prove our result, we recall one result of Schmidt’s subspace theorem for function field, due to  Wang  \cite{JW}. 

 Let $\mathcal{H}=\{H_1,\ldots,H_q\}$ be a set of hyperplanes in $\mathbb{P}^M(K)$ and $H_i(X)=\sum_{j=0}^Ma_{ij}X_j, a_{ij}\in K$ be the corresponding linear form with coefficients in $K$.
 
  The set of affine coordinates of (fixed) hyperplanes in $\mathcal{H}$  is defined by
  $$\mathcal{A}_\mathcal{H}:=\left\{\frac{a_{ij}}{a_{ii_0}}|1\leq i\leq q, 0\leq j\leq M\right\}$$
  where $a_{ii_0}$ is the first non-zero coefficients in $H_i(X),$ (i.e $a_{ij}=0$ if $j<i_0.$)
  
  Let $r$ be a positive integer. We let $V_{\mathcal{A}_\mathcal{H}}(r)$ be the finite dimensional $k-$vector space spanned by
  $$\left\{\prod_{v\in\mathcal{A}_\mathcal{H}}v^{n_v}|n_v\geq 0, \sum n_v=r\right\}$$
  over $k$.
\begin{lemma}\label{lem8} (see  \cite{JW})
Let $H_1,\dots, H_q$ be hyperplanes in $\mathbb P^{M}(K)$ and $S\subset M_K$ be a finite subset of $M_K$. Given $\varepsilon >0,$
one can find effectively a positive integer $r_{\varepsilon}$ and a constant $C_{\varepsilon}$ such that if 
${\bf x}=[x_0:\dots : x_M]\in \mathbb P^{M}(K) \setminus \cup_{i=1}^{q}H_i$ are linearly nondegenerate over $V_{{\mathcal{A}}_{\mathcal H}}(r_\epsilon)$, then the 
inequality holds
$$ \sum_{{\bf \mathfrak p}\in S}\max_{J}\sum_{j\in J} \lambda_{{\bf \mathfrak p}, H_j}({\bf x})\le (M+1+\varepsilon)h({\bf x})+
C_{\varepsilon},$$
and the maximum is taken over all subsets $J$ of $\{1, \dots, q\}$ such that the linear forms $H_j, j\in J$, are linearly independent over $K$. 
Furthermore, $r_{\varepsilon}$ depends only on the set ${\mathcal{A}}_{\mathcal H}$ and the constant $h(\mathcal H).$
\end{lemma}

\setcounter{equation}{0}

\begin{proof}[Proof of Theorem \ref{th1}]
We  can assume that $q>M+1.$ Let ${\bf x}=[x_0: \dots: x_M]: \Lambda \to  \mathbb P^{M}(K)$ be a sequence of points indexed by 
$\Lambda.$ Since ${\bf x}$  is linearly nondegenerate with respect to $\mathcal H=\{H_1, \dots, H_q\},$ this implies that there exists
an infinite subset $A\subset \Lambda$ such that $x_0|_A, \dots, x_M|_A$ are linearly independent over $\mathcal R_{A}.$ If $B$ is any infinite subset of $A$ then $B$ is still coherent and $x_0,\ldots,x_M$ are still linearly independent over $\mathcal{R}_B$. Therefore, we may freely pass to an infinite subsequences.

 We choose $a_{j,0}, \dots, a_{j,M}$ such that for every $\alpha \in A$, $H_j(\alpha), j=1, \dots, q$ is defined by the equation 
$$ a_{j,0}(\alpha)x_0+\dots+a_{j,M}(\alpha)x_M=0. $$
By  coherence of $A$, for each $j$, there exists $a_{j,j_0}(\alpha)$, one of the coefficients in $H_j(\alpha)$, such that
 $a_{j, j_0}(\alpha)\ne 0$
 for all but finitely many $\alpha \in A.$ Fix this $a_{j,j_0}$, we see that $\xi_{j,l}=\dfrac{a_{j, l}}{a_{j, j_0}}
$ defines an element of
 $\mathcal R_{A,\{H_j\}_{j=1}^{q}}$. Since $h(H_j(\alpha))=o(h(x(\alpha)))$, we have $\xi_{j,l}\in\mathcal{K}_x$. 
 Let $L_j:\Lambda\longrightarrow \mathbb{P}^M(K)^*$ be a map given by 
 $$L_j(\alpha)=\xi_{j,0}(\alpha)x_0+\ldots+\xi_{j,M}(\alpha)x_M$$ 
By passing to an infinite subsequences, for each ${\bf \mathfrak p}\in S$, there exists a set 
$J({\bf\mathfrak p})=\{j_0({\bf\mathfrak p}), \dots, j_{M}({\bf\mathfrak p})\} \subset \{1, \dots, q\}$ such that 
\begin{align*}
 \textup{ord}_{\bf \mathfrak p}(L_{j_0({\bf\mathfrak p})}(\alpha)({\bf x}(\alpha)))\ge \dots& \ge  
\textup{ord}_{\bf \mathfrak p}(L_{j_{M}({\bf\mathfrak p})}(\alpha)({\bf x}(\alpha)))\\
&\ge \max_{j\not \in \{j_0({\bf\mathfrak p}), \dots, j_{M}({\bf\mathfrak p})\}}  \textup{ord}_{\bf \mathfrak p}
(L_{j}(\alpha)({\bf x}(\alpha)))
\end{align*}
for all $\alpha\in A.$
We have
\begin{align}\label{ctb1}
L_{j_i({\bf\mathfrak p})}(\alpha)({\bf x}(\alpha))
=\xi_{{j_i({\bf\mathfrak p})},0}(\alpha)x_0(\alpha)+\dots+\xi_{{j_i({\bf\mathfrak p})},M}(\alpha)x_M(\alpha),
\hspace{0.2 cm}i=0, \dots, M.
\end{align}
Since $H_{j_i({\bf\mathfrak p})}(\alpha), i=0, \dots, M$ are in general position, by solving the system of linear 
equations (\ref{ctb1}), we have
\begin{align}\label{ctb2}
x_i(\alpha)=
\overset{\sim}\xi_{{j_i({\bf\mathfrak p})},0}(\alpha)L_{j_0({\bf\mathfrak p})}(\alpha)({\bf x}(\alpha))+\dots+\overset{\sim}
\xi_{{j_i({\bf\mathfrak p})},M}(\alpha)L_{j_{M}({\bf\mathfrak p})}(\alpha)({\bf x}(\alpha))
\end{align}
for all $i=0, \dots, M,$  where $\Big(\overset{\sim}\xi_{{j_i({\bf\mathfrak p)},l}}(\alpha)\Big)_{i,l=0}^{M}$
is the inverse of matrix $\Big(\xi_{{j_i({\bf\mathfrak p})},l}(\alpha)\Big)_{i,l=0}^{M}.$ Notice that $\overset{\sim}\xi_{{j_i({\bf\mathfrak p)},l}}\in \mathcal{K}_x.$

Then, from (\ref{ctb2}), we have
\begin{align*}
{\textup{ord}}_{\bf \mathfrak p}(x_i(\alpha))&\ge
\min\{\textup{ord}_{\bf \mathfrak p}(L_{j_0({\bf\mathfrak p})}(\alpha)({\bf x}(\alpha))), \dots,
 \textup{ord}_{\bf \mathfrak p}(L_{j_{M}({\bf\mathfrak p})}(\alpha)({\bf x}(\alpha)))\}\\
&+\min\{\textup{ord}_{\bf \mathfrak p}(\overset{\sim} \xi_{{j_i({\bf\mathfrak p})},0}(\alpha)), \dots,
\textup{ord}_{\bf \mathfrak p}(\overset{\sim} \xi_{{j_i({\bf\mathfrak p})},M}(\alpha))\},
\end{align*}
for all $i=0, \dots, M.$ Thus, we have
\begin{align}\label{ctb3}
e_{\bf \mathfrak p}({\bf x}(\alpha))&\ge
\min\{\textup{ord}_{\bf \mathfrak p}(L_{j_0({\bf\mathfrak p})}(\alpha)({\bf x}(\alpha))), \dots,
 \textup{ord}_{\bf \mathfrak p}(L_{j_{M}({\bf\mathfrak p})}(\alpha)({\bf x}(\alpha)))\}\notag\\
&+\min\{\textup{ord}_{\bf \mathfrak p}(\overset{\sim} \xi_{{j_i({\bf\mathfrak p})},0}(\alpha)), \dots,
\textup{ord}_{\bf \mathfrak p}(\overset{\sim} \xi_{{j_i({\bf\mathfrak p})},M}(\alpha))\}.
\end{align}
Hence, we have
\begin{align}\label{ctb4}
\sum_{j=1}^{q}(e_{{\bf\mathfrak p}}({\bf x}(\alpha))-&\textup{ord}_{\bf \mathfrak p}(L_j(\alpha)(x(\alpha)))=
\sum_{i=0}^{M}(e_{{\bf\mathfrak p}}({\bf x}(\alpha))-\textup{ord}_{\bf \mathfrak p}(L_{j_i({\bf \mathfrak p})}(\alpha)({\bf x}(\alpha))))\notag\\
&+\sum_{\beta_j\not \in J({\bf \mathfrak p}, \alpha)}(e_{{\bf\mathfrak p}}({\bf x}(\alpha))-\textup{ord}_{\bf \mathfrak p}(L_{\beta_j}(\alpha)({\bf x}(\alpha))))\notag\\
&\ge \sum_{i=0}^{M}(e_{{\bf\mathfrak p}}({\bf x}(\alpha))-\textup{ord}_{\bf \mathfrak p}(L_{j_i({\bf \mathfrak p})}(\alpha)({\bf x}(\alpha))))\notag\\
&+(q-M-1)\min\{\textup{ord}_{\bf \mathfrak p}(\overset{\sim} \xi_{{j_i({\bf\mathfrak p})},0}(\alpha)), \dots,
\textup{ord}_{\bf \mathfrak p}(\overset{\sim} \xi_{{j_i({\bf\mathfrak p})},M}(\alpha))\}
\end{align}
Combining with $\overset{\sim} \xi_{{j_i({\bf\mathfrak p})},l}\in \mathcal{K}_x,$, this implies 
\begin{align}\label{ctb5}
\sum_{j=1}^{q}\lambda_{{\bf\mathfrak p}, H_j(\alpha)}({\bf x}(\alpha))&\le \sum_{j=0}^{M}(\textup{ord}_{\bf \mathfrak p}(L_{j_i({\bf \mathfrak p})}(\alpha)({\bf x}(\alpha)))-e_{{\bf\mathfrak p}}({\bf x}(\alpha)))\deg {\bf \mathfrak p}
-(\sum_{j=1}^{q}e_{{\bf\mathfrak p}}(L_j(\alpha)))\deg {\bf \mathfrak p}\notag\\
&+o(h(x(\alpha)))
\end{align}

Let $\mathcal L(s)$ be the vector space generated by
$$ \Big\{\xi_{1,0}^{M_{1,0}}\dots \xi_{q,0}^{M_{q,0}}\xi_{1,M}^{M_{1,M}}\dots \xi_{q, M}^{M_{q, M}}
\Big| M_{i,j}\in \N, \sum_{i=1}^{q}\sum_{j=0}^{M}M_{i,j}=s\Big\}.$$
We see that $\mathcal L(s)\subset \mathcal L(s+1).$ For each $s$, let $l(s) =\dim \mathcal L(s).$ As noted by Steinmetz, we have
$$ \lim\inf_{s\to \infty}\dfrac{l(s+1)}{l(s)} =1.$$
Then, for $\delta >0$, we may find $s\in\mathbb{Z}$ such that
$$ l(s+1) \le (1+\delta)l(s)$$

Fix such an $s$. Let $\{b_1, \dots, b_{l(s+1)}\}$ be a basic of $\mathcal L(s+1)$ such that $\{b_1,\dots, b_{l(s)}\}$ is a 
basic of $\mathcal L(s).$ Then 
\begin{align}\label{ctb51}
\Big(b_{\mu}x_{\nu}\Big)_{\mu=1, \dots, l(s+1); \nu=0, \dots, M}
\end{align}
are linearly independent over $K.$
Indeed, we consider the following equality
$$ \sum_{\mu=1, \dots, l(s+1); \nu=0, \dots, M}\gamma_{\mu, \nu}b_{\mu}x_{\nu}=0,$$
where $\gamma_{\mu, \nu} \in K.$
This implies
$$\sum_{\nu=0}^{M}(\sum_{\mu=1}^{l(s+1)}\gamma_{\mu, \nu}b_{\mu})x_{\nu}=0.$$
Since $x_0|_A, \dots, x_M|_A$ are linearly independent over $K$ implies

\begin{align}\label{ctb6}
\sum_{\mu=1}^{l(s+1)} \gamma_{\mu, \nu} b_{\mu}=0,
\end{align}

for all $\nu=0, \dots, M.$
We note that $\{b_1, \dots, b_{l(s+1)}\}$ is a basic of $\mathcal L(s+1).$ Thus, from (\ref{ctb6}), we have
$\gamma_{\mu, \nu}=0$ for all $\mu=1, \dots, l(s+1)$, $\nu=0, \dots, M.$ 
For $j=1, \dots, q$, let $h_j\in \mathcal R_{A}$ define by
\begin{align}\label{ctb6a}
 h_j=\sum_{l=0}^{M}\xi_{j,l}x_l. 
\end{align}

For ${\bf\mathfrak p} \in S$ and for the chosen set $J(\mathfrak{p})=\{j_0({\bf\mathfrak p}), \dots, j_{M}({\bf\mathfrak p})\}
 \subset \{1, \dots, q\}$ satisfying the inequality (\ref{ctb5}), the $b_jh_{j_l({\bf\mathfrak p})}, j=1, \dots, l(s); l=0, \dots, 
M$ can be written uniquely as $K-$ linear combinations of the products $b_{\mu}x_{\nu}$ for $\mu=1, \dots, l(s+1),$ and $ \nu=0, \dots, M.$ In other word, there exists a matrix
$C({\bf \mathfrak p}) \in Mat_{(M+1)l(s){\times}{(M+1)l(s+1)}}(K)$ (the entries are not just in $\mathcal{R}_A$) such that
\begin{align}\label{ctb6a1}
C({\bf \mathfrak p})\cdot \left ( \begin{matrix}
b_1x_0\cr
\cdot\cr
\cdot\cr
b_{l(s+1)}x_0\cr
\cdot\cr
\cdot\cr
b_1x_M\cr
\cdot\cr
\cdot\cr
b_{l(s+1)}x_M
\end{matrix}\right)=\left ( \begin{matrix}
b_1h_{j_0({\bf\mathfrak p})}\cr
\cdot\cr
\cdot\cr
b_{l(s)}h_{j_0({\bf\mathfrak p})}\cr
\cdot\cr
\cdot\cr
b_1h_{j_{M}({\bf\mathfrak p})}\cr
\cdot\cr
\cdot\cr
b_{l(s)}h_{j_{M}({\bf\mathfrak p})}
\end{matrix}\right).
\end{align}
For ${\bf\mathfrak p} \in S, l=0, \dots, M$ and $j=1, \dots, l(s)$, let  $\overset{\sim}H_{l, j}({\bf\mathfrak p})$ be  the hyperplanes
in $\mathbb P^{(M+1)l(s+1)-1}(K)$ defined by the corresponding row in $C({\bf\mathfrak p}).$ This means that if  $c_{ij}(\mathfrak{p})$ denote the elements of
$C({\bf\mathfrak p}),$ then
\begin{align*}
\overset{\sim}H_{l, j}({\bf\mathfrak p})=\{[&y_{1,0}:\dots: y_{l(s+1),0}:\dots :y_{1, M}: \dots : y_{l(s+1), M}]\in 
\mathbb P^{(M+1)l(s+1)-1}(K)\\
&|c_{ll(s)+j, 1}({\bf\mathfrak p})y_{1,0}+\dots+c_{ll(s)+j, (M+1)l(s+1)}({\bf\mathfrak p})y_{l(s+1),M}=0\}.
\end{align*}
Set $$\overset{\sim}L_{l, j}({\bf\mathfrak p})=c_{ll(s)+j, 1}({\bf\mathfrak p})y_{1,0}+\dots+c_{ll(s)+j, (M+1)l(s+1)}({\bf\mathfrak p})y_{l(s+1),M}$$ be linear form defining $\overset{\sim}H_{l, j}({\bf\mathfrak p})$

Because, $H_1(\alpha), \dots, H_q(\alpha)$ are in general position for each $\alpha$, and  $x_0, \dots, x_M$ are linearly independent over
$\mathcal R_{A},$ this implies that $h_{j_l({\bf\mathfrak p})}, l=0, \dots, M$  are linearly independent over $\mathcal R_{A}.$
Hence, by the choice of $b_1, \dots, b_{l(s)}$ we have $(b_jh_{j_l({\bf\mathfrak p})})_{j=1, \dots, l(s)}; l=0, \dots, M$ are linearly independent
over $K.$ Thus, $(\overset{\sim}H_{l, j}({\bf\mathfrak p}))_{l=0, \dots, M; j=1, \dots, l(s)}$  are linearly independent over $K$ 
for each ${\bf\mathfrak p} \in S.$

We denote $K^{A}$ by the set maps $A \to K.$  Then $K^{A}$ is a ring which contains $k$ as a subring (we embed $K$ into $K^{A}$ as 
 constant functions). By  remark \ref{re:1.3}, after removing finitely many elements from $A,$ we may suppose that $a_{j,j_0}(\alpha) \ne 0$
for all $\alpha \in A$, for some $j_0 \in \{0, \dots, M\}.$ Then, $\xi_{j,l}^{\#}=\dfrac{a_{j, l}}{a_{j, j_0}}
$ belongs to $K^{A}.$ We know that $\mathcal L(s)$
is generated by the monomials in the $\xi_{j,l}$ and $b_i$ are polynomials in several variables of $\xi_{j,l},$ 
it follows that they define corresponding
 elements $b_i^{\#} \in K^{A}$, $i=1, \dots, l(s+1).$
 Then, after
removing  finitely many elements of $A$, (\ref{ctb6a1})  is true in $K^{A}.$ From now on, we will work in $K^{A}$ instead of $\mathcal R_{A}$ without
the symbol $\#.$

For each $\alpha \in A$, let $P(\alpha)\in \mathbb P^{(M+1)l(s+1)-1}(K)$ be defined by 
\begin{align}\label{ct7a}
P(\alpha)=&[\dots: P_i(\alpha):\dots]=[b_1(\alpha)x_0(\alpha): \dots : b_{l(s+1)}(\alpha)x_0(\alpha): b_1(\alpha)x_1(\alpha):\notag\\
&\dots : b_{l(s+1)}x_1(\alpha): \dots: b_1(\alpha)x_M(\alpha): \dots : b_{l(s+1)}(\alpha)x_M(\alpha)].
\end{align}
We add $(M+1)(l(s+1)-l(s))$ hyperplanes 
$\overset{\sim}H_{t}({\bf\mathfrak p}), t=(M+1)l(s)+1, \dots, (M+1)l(s+1)$ to the family of hyperplanes 
$\overset{\sim}H_{l, j}({\bf\mathfrak p})$, $l=0, \dots, M, j=1, \dots, l(s)$ such that the new family is still linearly independent. Denote by 
$\mathcal {\overset{\sim} H}(\mathfrak{p})=\{\overset{\sim}H_{l, j}({\bf\mathfrak p}) \cup \overset{\sim}H_{t}({\bf\mathfrak p})\}.$

Notice that since $(b_\mu x_\nu)$ are linearly independent over $K$ then we have $P(\alpha)$ is linearly independent over $V_{A_{\mathcal {\overset{\sim} H}(\mathfrak{p})}}\subset K$. Thus, we can apply  Lemma \ref{lem8} to  the family of hyperplanes $\mathcal {\overset{\sim} H}(\mathfrak{p})$  and the points $P(\alpha)\in \mathbb P^{(M+1)l(s+1)-1}(K)$. Thus, there is a finite collection $\mathcal{L}$ of proper linear subspaces of $\mathbb P^{(M+1)l(s+1)-1}(K)$ such that
$$ \sum_{{\bf \mathfrak p}\in S}\sum_{l=0}^{M}\sum_{j=1}^{l(s)} \lambda_{{\bf \mathfrak p}, \overset{\sim} H_{{l, j({\bf \mathfrak p})}}}(P(\alpha))\le((M+1)l(s+1)+\delta)h(P(\alpha))+O(1).
$$
for all $\alpha$ such that $P(\alpha)\not\in\cup_{L\in\mathcal{L}}L.$ By remark \ref{re:1.3}, we may pass to an infinite subsequence satisfying $P(\alpha)\not\in\cup_{L\in\mathcal{L}}L $ for all $\alpha$ in the subsequence. 

Therefore,\begin{align}\label{ct9b}
 \sum_{{\bf \mathfrak p}\in S}\sum_{l=0}^{M}\sum_{j=1}^{l(s)} \lambda_{{\bf \mathfrak p}, \overset{\sim} 
H_{{l, j({\bf \mathfrak p})}}}(P(\alpha))
\le ((M+1)l(s+1)+\delta)h(P(\alpha))+O(1).
\end{align}
for all $\alpha$ in $A.$

We have \begin{align}\label{ct10b}
h(P(\alpha))&=-\sum_{{\bf \mathfrak p}\in M_K}\min_{i}\textup{ord}_{\bf \mathfrak p}(P_i(\alpha))\deg {\bf\mathfrak p}\notag\\
&=-\sum_{{\bf \mathfrak p}\in M_K}\min_{0\le i\le M}\textup{ord}_{\bf \mathfrak p}(x_i(\alpha))\deg {\bf\mathfrak p}-
\sum_{{\bf \mathfrak p}\in M_K}\min_{1\le j\le l(s+1)}\textup{ord}_{\bf \mathfrak p}(b_j(\alpha))\deg {\bf\mathfrak p}\notag\\
&=h({\bf x}(\alpha))+o(h({\bf x}(\alpha))).
\end{align}
 We compute the Weil function $\lambda_{{\bf \mathfrak p}, \overset{\sim} H_{{l, j({\bf \mathfrak p})}}}$ at the points $(P(\alpha)).$ We see that
\begin{align*}
&\lambda_{{\bf \mathfrak p}, \overset{\sim} H_{{l, j({\bf \mathfrak p})}}}(P(\alpha))
=(\textup{ord}_{\bf \mathfrak p}\overset{\sim} L_{{l, j({\bf \mathfrak p})}}(P(\alpha))-e_{\bf \mathfrak p}(P(\alpha))-
e_{\bf \mathfrak p}(\overset{\sim} L_{{l, j({\bf \mathfrak p})}})) \deg {\bf \mathfrak p}\\
&=\left(\textup{ord}_{\bf \mathfrak p}b_j(\alpha)h_{j_l({\bf\mathfrak p})}(\alpha)-e_{\bf \mathfrak p}(P(\alpha))-
e_{\bf \mathfrak p}(\overset{\sim} L_{{l, j({\bf \mathfrak p})}})\right)\deg\mathfrak{p}\\
&=\left(\textup{ord}_{\bf \mathfrak p}b_j(\alpha)+\textup{ord}_{\bf \mathfrak p}h_{j_l({\bf\mathfrak p})}(\alpha)
-e_{\bf \mathfrak p}({\bf x}(\alpha)
-\min_{1\le j\le l(s+1)}\textup{ord}_{\bf \mathfrak p}b_j(\alpha)\right)\deg{\bf \mathfrak p}
-e_{\bf \mathfrak p}(\overset{\sim} L_{{l, j({\bf \mathfrak p}}})) \deg {\bf \mathfrak p}\\
&=\left(\textup{ord}_{\bf \mathfrak p}h_{j_l({\bf\mathfrak p}, \alpha)}(\alpha)
-e_{\bf \mathfrak p}({\bf x}(\alpha))-e_{\bf \mathfrak p}( L_{{l, j({\bf \mathfrak p})}}(\alpha))\right)\deg {\bf \mathfrak p}\\
&+\left(e_{\bf \mathfrak p}( L_{{l, j({\bf \mathfrak p})}}(\alpha))+\textup{ord}_{\bf \mathfrak p}b_j(\alpha)-\min_{1\le j\le l(s+1)}\textup{ord}_{\bf \mathfrak p}(b_j(\alpha) \deg {\bf \mathfrak p}\right)
 \deg {\bf \mathfrak p}-e_{\bf \mathfrak p}(\overset{\sim} L_{{l, j({\bf \mathfrak p}}})) \deg {\bf \mathfrak p}\\
&=\lambda_{{\bf \mathfrak p},H_{j_{l({\bf \mathfrak p}, \alpha)}}}({\bf x}(\alpha))+o(h({\bf x}(\alpha)))+O(1).
\end{align*}
From (\ref{ct9b}), we have
\begin{align}\label{ct9b1}
 \sum_{{\bf \mathfrak p}\in S}l(s)\sum_{l=0}^{M}\lambda_{{\bf \mathfrak p},H_{j_{l({\bf \mathfrak p}, \alpha)}}}
({\bf x}(\alpha)) = \sum_{{\bf \mathfrak p}\in S}\sum_{l=0}^{M}\sum_{j=1}^{l(s)} \lambda_{{\bf \mathfrak p}, \overset{\sim}H_{{l, j({\bf \mathfrak p})},{\bf \mathfrak p}}}(P(\alpha)) \notag\\
\le ((M+1)l(s+1)+\delta)h(P(\alpha))+o(h({\bf x}(\alpha)))+O(1),
\end{align}
for all $\alpha\in A.$
 
 Together with (\ref{ctb5}), we have
 \begin{align}
 \sum_{{\bf \mathfrak p}\in S}l(s)\sum_{l=0}^{q}\lambda_{{\bf \mathfrak p},H_{j_{l({\bf \mathfrak p}, \alpha)}}}
({\bf x}(\alpha)) = \sum_{{\bf \mathfrak p}\in S}\sum_{l=0}^{M}\sum_{j=1}^{l(s)} \lambda_{{\bf \mathfrak p}, \overset{\sim}H_{{l, j({\bf \mathfrak p})},{\bf \mathfrak p}}}(P(\alpha)) \notag\\
\le ((M+1)l(s+1)+\delta)h(P(\alpha))+o(h({\bf x}(\alpha)))+O(1),
\end{align}
for all $\alpha\in A.$ Therefore,
 \begin{align*}
\sum_{j=1}^{q}\lambda_{{\bf\mathfrak p}, H_j(\alpha)}({\bf x}(\alpha))
&\le\left(\dfrac{(M+1)l(s+1)}{l(s)}+\dfrac{\delta}{l(s)}\right)h(P(\alpha))+o(h({\bf x}(\alpha))\\
&\le (M+1+\varepsilon)h({\bf x}(\alpha))+O(1),
\end{align*}
 for all $\alpha \in A.$
This completes the proof of Theorem \ref{th1}.
\end{proof}

\end{document}